\documentclass[12pt,a4paper]{amsart}
\usepackage{amsmath,amsthm,amsfonts,amssymb,mathrsfs}
\usepackage{subfigure} 
\usepackage{graphicx}
\usepackage{color}

\newcommand\Tm{T_{max}}
\newcommand\R{\mathbb{R}}

\def\bOmega{\overline \Omega}

\def\u{\bar{u}}
\def\v{\bar{v}}

\newtheorem{theorem}{Theorem}

\newtheorem{lemma}[theorem]{Lemma}

\numberwithin{equation}{section}
\numberwithin{theorem}{section}
\numberwithin{figure}{section}

\theoremstyle{remark}
\newtheorem{rem}[theorem]{Remark}

\theoremstyle{remark}

\usepackage{delarray}
\usepackage{enumerate}

\begin{document}

\title[Diffusion-driven blowup]{Diffusion-driven blowup of nonnegative \\
solutions 
to reaction-diffusion-ODE systems}

\author[A.~Marciniak-Czochra]{Anna Marciniak-Czochra}
\address[A. Marciniak-Czochra]{
Institute of Applied Mathematics,  Interdisciplinary Center for Scientific Computing (IWR) and BIOQUANT, University of Heidelberg, 69120 Heidelberg, Germany}
\email{anna.marciniak@iwr.uni-heidelberg.de}
\urladdr {http://www.biostruct.uni-hd.de/}

\author[G.~Karch]{Grzegorz Karch}
\address[G. Karch]{
 Instytut Matematyczny, Uniwersytet Wroc\l awski,
 pl. Grunwaldzki 2/4, 50-384 Wroc\-\l aw, Poland}
\email{grzegorz.karch@math.uni.wroc.pl}
\urladdr{http://www.math.uni.wroc.pl/~karch}

\author[K.~Suzuki]{Kanako Suzuki}
\address[K. Suzuki]{
College of Science, Ibaraki University,
2-1-1 Bunkyo, Mito 310-8512, Japan}
\email{kanako.suzuki.sci2@vc.ibaraki.ac.jp}

\author[J. Zienkiewicz]{Jacek Zienkiewicz}
\address[J. Zienkiewicz]{
 Instytut Matematyczny, Uniwersytet Wroc\l awski,
 pl. Grunwaldzki 2/4, 50-384 Wroc\-\l aw, Poland}
\email{jacek.zienkiewicz@math.uni.wroc.pl}

\date{\today}


\begin{abstract}
In this paper we provide an example of a class of two reaction-diffusion-ODE equations with homogeneous Neumann boundary conditions, in which Turing-type instability not only destabilizes constant steady states but also induces blow-up of nonnegative spatially heterogeneous solutions. 
Solutions of this problem preserve nonnegativity and uniform boundedness of the total mass. 
Moreover,  for the corresponding  system with two non-zero diffusion coefficients, all nonnegative solutions
are global in time. 
We prove that a removal of diffusion in one of the equations leads to a finite-time blow-up of some nonnegative spatially heterogeneous solutions.

\medskip

{\bf Keywords:}  reaction-diffusion equations; Turing  instability;  blow-up of solutions.
\end{abstract}
\maketitle

\section{Introduction}

One of the major issues in study 
of reaction-diffusion equations describing 
pattern formation in biological or chemical systems
is understanding of the mechanisms of pattern selection,
{\it  i.e.} of generation of stable patterns.  Classical models of the pattern formation are based on  {\it diffusion-driven instability (DDI)} of constant stationary solutions, which leads to emergence of stable patterns around this state. 
Such {\it close-to-equilibrium} patterns are regular and spatially periodic stationary solutions and their shape depend on a scaling coefficient related to 
the ratio between diffusion parameters. They are  called {\it Turing patterns} after the seminal paper of Alan Turing \cite{T52}. 

Interestingly, a variety of possible patterns increases when 
some diffusion coefficient vanish, {\it i.e.}~considering reaction-diffusion equations coupled to ordinary differential equations (ODEs). Such models arise, for example, when studying a coupling of diffusive processes with processes which are localized in space, such as growth processes \cite{MCK06,MCK07,MCK08,PCH12} or intracellular signaling  \cite{Hock13,KBH12,MC03,USOO06}. Their dynamics appear to be very different from that of classical reaction-diffusion models.

To understand the role of non-diffusive components in a pattern formation process, we
focus on systems involving a single reaction-diffusion equation coupled to an ODE. It is an interesting case, since a scalar reaction-diffusion equation 
(in a bounded, convex domain and the Neumann boundary conditions) 
cannot exhibit stable spatially heterogenous patterns \cite{CaHo}. Coupling it to an ODE fulfilling 
an {\it autocatalysis condition} at the equilibrium leads to DDI.  However, in such a case, all regular Turing patterns are unstable, because  the same mechanism which destabilizes constant solutions, destabilizes also all continuous spatially heterogeneous stationary solutions,  \cite{MKS13,MKS14}. 
 This instability result  holds also for discontinuous patterns in case of a specific class of nonlinearities, see also \cite{MKS13,MKS14}. 
Simulations of different models of this type indicate a formation of  dynamical, multimodal, and apparently irregular and unbounded structures, the shape of which depends strongly on initial conditions \cite{HMC13,MCK07,MCK08,PCH12}.

In this work, we attempt to make a next  step towards 
understanding properties of solutions of reaction-diffusion-ODE systems. We focus on a specific example exhibiting {diffusion-driven instability}.  
We consider the following system of equations
\begin{align}
u_t  &= d \Delta u -au +u^pf(v),&     \text{for}\quad &x\in \Omega, \;\quad t>0, && \label{eq1}\\
v_t  &=   D \Delta v-bv - u^pf(v)+\kappa&  \text{for}\quad &x\in \Omega, \quad t>0,&&\label{eq2}
\end{align}
in a bounded domain $\Omega\subset \R^n$ with a sufficiently regular boundary $\partial\Omega$.
In equations \eqref{eq1}-\eqref{eq2},  an arbitrary 
function $f=f(v)$ satisfies
\begin{equation}\label{as:f}
 f\in C^1([0,\infty)), \quad  f(v)>0  \quad \text{for}\quad  v>0, 
\quad \text{and}\quad 
f(0)=0.
\end{equation}
 
Moreover, we fix the constant parameters in 
\eqref{eq1}-\eqref{eq2} such that 
\begin{equation}\label{par}
d \geq 0, \quad D>0, \quad p>1, \quad 
 \quad a,b\in (0,\infty), \quad   \kappa \in [0,\infty).
\end{equation}
We supplement system \eqref{eq1}-\eqref{eq2} with the homogeneous Neumann boundary conditions
\begin{equation}\label{N}
\frac{\partial u}{\partial{n}}  =  0 \;
\text{ (if $d>0$})\quad \text{and}\quad  \frac{\partial v}{\partial{n}}=0 \qquad \text{for}\quad x\in \partial\Omega, \quad t>0,
\end{equation}
and with  bounded, nonnegative,  and continuous initial data
\begin{eqnarray}\label{ini}
&&u(x,0)  =   u_{0}(x),\qquad v(x,0)  =  v_{0}(x)
\qquad \text{for}\quad x\in \Omega.
\end{eqnarray}

%
%


As already mentioned above, 
if the diffusion in equation \eqref{eq1} is equal to zero,
all regular stationary solutions to such reaction-diffusion-ODE problems are unstable,  see \cite{MKS14}
for the results in the case of more general equations.
In this work, 
we show   that dynamics of solutions to
the initial-boundary value problem \eqref{eq1}-\eqref{ini}
may change  drastically when  $d>0$ in equation \eqref{eq1} is replaced by $d =0$.
More precisely, the following scenario is valid.

\begin{itemize}

\item For non-degenerate diffusion coefficients 
 $d>0$ and $D>0$, all nonnegative solutions to problem 
\eqref{eq1}-\eqref{ini} are global-in-time.
This result has been proved by other authors and, for the reader convenience, we 
discuss it in Section \ref{sec:global}, see Remark \ref{rem:global}.

\item If $d =0$ and $D>0$ ({\it i.e.}~we consider an ordinary differential equations coupled with a reaction-diffusion equation),  there are 
solutions to problem \eqref{eq1}-\eqref{ini} which blow-up in a finite time
and at one point only. This is the main result of this work, proved in  
 Theorem~\ref{thm:blowup}, below.
\end{itemize}

Let us emphasize some consequences of these results.

\begin{rem}[Diffusion-induced blow-up of nonnegative solutions] \label{rem:blow}
Nonnegative solutions to the following initial value problem for the system of ordinary differential equations:
\begin{align}\label{qeq}
&\frac{d}{dt} \u  = -a\u +\u^pf(\v), &&
\frac{d}{dt} \v =  -b\v - \u^pf(\v)+\kappa,&\\
\label{qini}
&\u(0)=\u_0\geq 0, && \v(0)=\v_0\geq 0.&
\end{align}
are global-in-time and bounded on $[0,\infty)$, see Remark \ref{rem:kinetic} below. 
On the other hand, by Theorem \ref{thm:blowup}  below, there are nonconstant initial conditions such that 
the corresponding solution to the reaction-diffusion-ODE problem 
\eqref{eq1}-\eqref{ini} with $d=0$ and $D>0$
  blows up 
at one point and in a finite time.
This is a large class of examples, where 
the appearance of 
a diffusion 
in one  equation leads to a  blow-up of nonnegative solutions. 
First  example of one reaction-diffusion equation coupled with one ODE, where 
some solutions blow up due to a diffusion, 
appeared in 1990 in the paper by  
Morgan  \cite{M90}. Another reaction-diffusion-ODE system  was 
given by Guedda and Kirane \cite{GK98}.
These examples are  discussed  in detail 
in the survey  paper \cite{FN05} as well as in the monograph \cite[Ch.~33.2]{QS07}.
Here, let us also mention that a one point blow-up result, analogous to that one in
 Theorem~\ref{thm:blowup} but  for another 
 reaction-diffusion-ODE system (with 
``activator-inhibitor'' nonlinearities) has been recently obtained by us  
in \cite{KSZ14}.
\qed
\end{rem}

\begin{rem}
It is much more difficult to provide a blow up of solutions 
in a  system of reaction-diffusion equations with nonzero  diffusion coefficients 
in both  equations, rather than in only one (as in Remark \ref{rem:blow}), 
especially in the case of systems with a 
 good ``mass behavior" as discussed in Remark \ref{rem:mass}.
First such an example was discovered by  Mizoguchi {\it et al.}
\cite{MNY98}, where  the term {\it ``diffusion-induced blow-up''} was introduced.
Another system of reaction-diffusion equations with such a property, supplemented   
with non-homogeneous Dirichlet conditions, was proposed by 
Pierre and Schmitt \cite{PS97,PS00}.  
We refer the reader to 
the survey  paper \cite{FN05} and to  the monograph \cite[Ch.~33.2]{QS07}
for more such examples and for additional comments.
\qed
\end{rem}

At the end of this introduction,
 we would like to emphasize that the model
\eqref{eq1}-\eqref{ini} can be found in literature in context of  several
applications. Let us mention a few of them.
 For $p=2$,  $f(v)=v$, and suitably chosen coefficients, 
we obtain either the,  so-called, {\it Brussellator} appearing
in the modeling of chemical morphogenetic processes
(see {\it e.g.} \cite{T52,NP77}),   the 
{\it Gray-Scott model} (also known as a 
{\it model of glycolysis}, see \cite{GS85,GS90})
or the {\it Schnackenberg model} (see \cite{S79} and \cite[Ch.~3.4]{M2}).
Recent mathematical results, as well as  several other references 
on reaction-diffusion equations with such nonlinearities and with $d>0$ and $D>0$, may be found in, {\it e.g.}, 
the monographs \cite{M2, QS07,R84} and in the papers
\cite{P10,Y08,YZ12}.
Let us close this introduction by a remark that 
we assume in this work that $a>0$ and $b>0$ for simplicity of the exposition, however, our blowup results can be easily modified to the case of arbitrary $a\in \R$ and $b\in \R$.

\section{Global-in-time solutions for reaction-diffusion system}\label{sec:global}

Results gathered in this section has been proved already by other authors and we recall them for the completeness of the exposition.

First, we recall   that problem \eqref{eq1}-\eqref{ini}
 supplemented with nonnegative initial data 
$u_0,v_0\in L^\infty(\Omega)$ has a unique,  
nonnegative local-in-time solution $(u(x,t), v(x,t))$.
Here, it suffices to rewrite it in 
the usual integral (Duhamel) form 
\begin{align}
&u(t) = e^{t(d \Delta -aI)}u_0+
 \int_0^t e^{(t-s)(d \Delta -aI)} \big(u^pf(v)\big)(s)\,ds, &&\label{duh1}\\
&v(t) = e^{t(D\Delta -bI)}v_0
-  \int_0^t e^{(t-s)(D\Delta -bI)} \big(u^pf(v)\big)(s)\,ds+ \int_0^t e^{(t-s)(D\Delta -bI)}\kappa\,ds,&& \label{duh2}
\end{align} 
where $\left\{e^{t(d \Delta-aI)}\right\}_{t\geq 0}$ 
is the semigroup of linear operators on $L^q(\Omega)$ generated by $d \Delta-aI$ with
the Neumann boundary conditions.
Since the nonlinearities in equations \eqref{eq1}-\eqref{eq2}
are  locally Lipschitz continuous, 
the existence of a local-in-time unique solution to \eqref{duh1}-\eqref{duh2} is a consequence of the Banach contraction 
principle, see {\it e.g.}~either \cite[Thm.~1, p.~111]{R84}
or \cite{HMP87}.
Such a  solution is sufficiently regular for $t\in (0, \Tm)$, where $\Tm>0$ is 
the maximal time of its existence,  and satisfies problem 
\eqref{eq1}-\eqref{ini} in the classical sense.
Moreover, this local-in-time solution 
$\big(u(x,t), v(x,t)\big)$ is nonnegative, either  by 
a maximum principle for parabolic equations if $d >0$ 
or for  reaction-diffusion-ODE systems
 if $d =0$, 
see {\it e.g.} \cite[Lemma 3.4]{MKS13} for similar considerations.

In the following, we review results on  the existence of global-in-time nonnegative  
solutions to problem     \eqref{eq1}-\eqref{ini} with the both 
$d>0$ and $D>0$.
We begin with the corresponding system of ODEs.

\begin{rem}\label{rem:kinetic}
It is a routine reasoning to show that  $x$-independent nonnegative 
solutions $(\u,\v)$ of problem \eqref{eq1}-\eqref{ini} 
are global-in-time and uniformly bounded.
Indeed, such a solution
$u=\u(t)$ and $v=\v(t)$ solves   the  Cauchy problem
for the system of ODEs  \eqref{qeq}-\eqref{qini}.
From equations \eqref{qeq}, we deduce  the differential inequality
\begin{equation}\label{q:ineq} 
\frac{d}{dt} \big( \u+\v)
\leq - \min\{a,b\} \big( \u+\v\big)+ \kappa
\end{equation}
which, after integration,  implies that the sum 
  $ \u(t)+\v(t)$ 
is bounded on the half-line $[0, \infty)$. Hence, since both functions are nonnegative, we obtain
$\sup_{t\geq 0} \u(t) <\infty$ and $ \sup_{t\geq 0} \v(t) <\infty.$
\qed
\end{rem}

\begin{rem}
A behavior of solutions the system of ODEs from \eqref{qeq} 
depends essentially on its parameters and,  
in the particular case of $p=2$ and $f(v)=v$, it  has been studied in several recent works, because it appears
in applications (see the discussion at the end of Introduction). 
For $a>0$ and $b>0$, this particular system  has the trivial stationary nonnegative solution $(\u,\v)=(0,\kappa/b)$
which is an asymptotically stable solution. If, moreover, $\kappa^2>4 a^2b$, we
have two other nontrivial nonnegative stationary solutions which satisfy the following system of equations 
$$
\u=\frac{a}{\v}
\qquad \text{and} \qquad 
  -b \v - \frac{a^2}{\v} +\kappa =0.
$$ 
Every such a constant  nontrivial  and {\it stable} solution of ODEs
is  an {\it unstable} solution of the reaction-diffusion-ODE 
problem \eqref{eq1}-\eqref{N}, which means that it has a DDI property
due to the autocatalysis
$ f_u(\u,\v)= - a+2 \bar u \bar v  =a>0.$
We have prove the latter property   in the  recent works
\cite{MKS13} and \cite{MKS14}, where such  instability  phenomena 
have been studied for a model of early carcinogenezis and for
 a general model of reaction-diffusion-ODEs, respectively. 
\qed
\end{rem}

\begin{rem}[Control of mass] \label{rem:mass}
A completely analogous reasoning as that one in Remark~\ref{rem:kinetic}
shows that total mass
$\int_\Omega \big( u(x,t)+v(x,t)\big)\,dx$
of each nonnegative solution 
to the reaction-diffusion problem \eqref{eq1}-\eqref{ini} 
with $d\geq 0$ and $D\geq 0$
does not blow up, 
and 
stays uniformly bounded in $t>0$. Indeed, 
it suffices to
sum up equations \eqref{eq1}-\eqref{eq2}, integrate over $\Omega$, 
and use the 
boundary condition to obtain the following counterpart 
of inequality \eqref{q:ineq}
\begin{align*}
\frac{d}{dt} \int_\Omega \big( u(x,t)+v(x,t)\big) \;dx 
&=- \int_\Omega \big( a u(x,t)+b v(x,t)\big) \;dx +\int_\Omega \kappa\;dx \\
&\leq - \min\{a,b\} \int_\Omega \big(  u(x,t)+v(x,t)\big)\;dx 
+ \kappa |\Omega|.
\end{align*}
Thus, the functions $u(\cdot, t)$ and  $v(\cdot, t)$ 
stay bounded in $L^1(\Omega)$ uniformly in time.
In the next section, we  show that
this {\it a priori} estimate is not sufficient to prevent the blow-up of
solutions in a finite time
in the case of $d=0$ and $D>0$ in problem 
 \eqref{eq1}-\eqref{ini}. 
\qed
\end{rem}

\begin{rem}[Global-in-time solutions]\label{rem:global}
Let $f\in C^1([0,\infty))$ be an  arbitrary function
satisfying conditions \eqref{as:f}. 
Assume that $d >0$ and $D>0$ and other parameters satisfy conditions \eqref{par}.
Then, for all nonnegative and continuous initial conditions $u_0,v_0\in L^\infty(\Omega)$, 
a unique nonnegative   solution of system \eqref{eq1}-\eqref{ini} exists for 
all $t\in (0,\infty)$.  
This result was  
proved by Masuda \cite{M83} and generalized by Hollis {\it et al.} 
 \cite{HMP87} as well as by  Haraux and Youkana \cite{HY88} 
(see also the surveys \cite{P02} and   \cite[Thm.~3.1]{P10}). 

Let us  briefly sketch  the proof  of the global-in-time existence of solutions 
for the reader convenience and for the completeness of exposition.
To  show that a local-in-time solution to integral equations \eqref{duh1}-\eqref{duh2}
can be continued  globally in time it suffices to   show {\it a priori} estimates
\begin{equation}\label{extend}
\sup_{t\in [0,\Tm)}\|u(t)\|_\infty<\infty \quad\text{and}\quad 
\sup_{t\in [0,\Tm)}\|v(t)\|_\infty<\infty \qquad \text{if} \quad \Tm<\infty.
\end{equation}
First, we  notice that, since $ u^pf(v)\geq 0$ for nonnegative $u$ and $v$, 
 the function $v(x,t)$ satisfies the 
inequalities
\begin{equation}\label{v:est0} 
0\leq  v(x,t)\leq \max\left\{\|v_0\|_\infty, \frac{\kappa}{b}\right\}
\qquad \text{for all} \quad (x,t)\in \Omega \times [0,\Tm),
\end{equation}
due to the comparison principle applied  to the parabolic equation 
\eqref{eq2}.
Thus,  the second inequality  in \eqref{extend}  is an
 immediate consequence of estimate  \eqref{v:est0}.

 To find an analogous estimate for $u(x,t)$, we observe that by equation \eqref{eq1}-\eqref{eq2}, we have 
$$
u_t  - d \Delta u +au =
- v_t  +  D \Delta v-bv +\kappa.
$$
Thus, using the Duhamel principle, we obtain
$$
u(t) = e^{t(d \Delta -aI)}u_0+
\big(-\partial_t  +  D \Delta -bI \big)
 \int_0^t e^{(t-s)(d\Delta -aI)} v(s)\,ds
+
 \int_0^t e^{(t-s)(d\Delta -aI)} \kappa \,ds.
$$
Since $u_0\in L^\infty(\Omega)$ and $v\in L^\infty 
\big(\Omega\times [0,\Tm)\big)$,
by a   standard $L^p$-regularity property of linear parabolic equations with the Neumann boundary conditions  
(see {\it e.g.} 
\cite[Ch.~III,~\S 10]{LSU}),
we obtain that 
$u\in L^q\big(\Omega\times [0,\Tm]\big)$ for each $q\in (1,\infty)$.
Using this property in equation \eqref{duh1}
and a well-known regularizing effect for linear parabolic equations
(\cite{LSU}),
we complete the proof of {\it a priori} estimate
$
\sup_{t\in [0,\Tm)}\|u(t)\|_\infty<\infty.
$
We refer the reader to \cite{P02,P10} for more details.
\qed
\end{rem}

\begin{rem}
If $\kappa=0$ in equation \eqref{eq2},
applying {\it e.g.} \cite[Theorem 2]{HMP87} we obtain that nonnegative solutions 
to problem \eqref{eq1}-\eqref{ini} with non-degenerate diffusions $d>0$ and $D>0$ 
are not only global-in-time (as stated in Remark \ref{rem:global}) but also 
uniformly bounded  on $\Omega\times [0,\infty)$. 
We do not know if
  this additional assumption on $\kappa$ is  necessary 
to show a uniform bound for solutions to this problem.
\qed
\end{rem}

\section{Blowup in a finite time for reaction-diffusion-ODE system}\label{sec:blowup}

Our main goal in this work is to show 
 that the result on the global-in-time existence of solutions 
to problem \eqref{eq1}-\eqref{ini} recalled in  Remark  \ref{rem:global}
is no longer  true if $d=0$.
 Thus, in the following, we consider the 
initial-boundary 
value problem for the reaction-diffusion-ODE system of the form
\begin{align}
&u_t  = -au +u^pf(v),&     \text{for}\quad &x\in\overline{\Omega}, \;\quad t\in [0,\Tm), && \label{oeq1}\\
&v_t  =    \Delta v-bv - u^pf(v)+\kappa&  \text{for}\quad &x\in \Omega, \quad t\in [0,\Tm),&&\label{oeq2}\\
& \frac{\partial v}{\partial n}=0
&  \text{on}\quad &\partial\Omega\times [0,\Tm), &&\label{oeq3}\\
&u(x,0)  =   u_{0}(x),\quad v(x,0)  =  v_{0}(x)
&  \text{for}\quad &x\in \Omega, \quad t\in [0,\Tm).&&\label{oini}
\end{align}
Here,  
without loss of generality, we assume that $0\in \Omega$,
where $\Omega\subset \R^n$ is an arbitrary bounded domain with a smooth boundary, and we rescale 
system \eqref{oeq1}-\eqref{oeq2} in such a way  that the  diffusion coefficient in equation \eqref{oeq2} is equal to one.

In the following theorem, we prove that if 
$u_0$ is concentrated around  an arbitrary point $x_0\in\Omega$ (we choose $x_0=0$, for simplicity)
and if $v_0(x)=\v_0$ is a constant  function, then the corresponding solution to
 problem \eqref{oeq1}-\eqref{oini}  blows up in a finite time.

\begin{theorem}\label{thm:blowup}
Assume that  $f\in C^1([0,\infty))$ satisfies $\inf_{v\geq R} f(v)>0$ for each $R>0$.
Let $p>1$ and $a,b,\kappa \in (0,\infty)$ be arbitrary. 
There exist numbers  $\alpha\in (0,1)$, $\varepsilon>0$,  and $R_0>0$
{\rm (}depending on parameters of problem  \eqref{oeq1}-\eqref{oini}
and determined in the proof{\rm )}
such that 
if  initial conditions $u_0,v_0\in C(\bOmega)$
satisfy 
\begin{align}
\label{as1:u0}
&0<u_0(x)<\Big(u_0(0)^{1-p}+2 \varepsilon^{-(p-1)} 
|x|^{\alpha}\Big)^{-\frac{1}{p-1}} 
&& \text{for all} \quad x\in \Omega \\
\label{as3:u0}
&u_0(0)\geq  \left(\frac{a}{(1-e^{(1-p)a})F_0}\right)^{\frac{1}{p-1}},
&& \text{where} \quad F_0 =\inf_{v\geq R_0} f(v),\\
\label{as:v0}
&v_0(x)\equiv \v_0>R_0>0
 && \text{for all} \quad x\in \Omega,
\end{align}
 then
 the corresponding solution to problem \eqref{oeq1}-\eqref{oini}
blows up at certain time $\Tm \leq 1$.
Moreover, the following uniform estimates are valid
\begin{equation}\label{est:uv}
0<u(x,t)<\varepsilon |x|^{-\frac{\alpha}{p-1}} \quad \text{and}\quad 
v(x,t)\geq R_0
\qquad \text{for all} \quad (x,t)\in \Omega\times [0,\Tm).
\end{equation}
\end{theorem}

\begin{rem}\label{rem:u0}
It follows  from assumption \eqref{as1:u0} that 
$$
0<u_0(x)< 2^{-\frac{1}{p-1}}
\varepsilon 
|x|^{-\frac{\alpha}{p-1}}\qquad 
 \text{for all} \quad x\in \Omega,
$$
for small $\varepsilon>0$.
On the other hand, assumption \eqref{as3:u0} requires $u_0(0)$ to be sufficiently large.
Both assumptions  mean that the function $u_0$ has to be  concentrated in a neighborhood of $x=0$.
\qed
\end{rem}

\begin{rem}
Notice that both inequalities in \eqref{est:uv} give us pointwise estimates of $u(x,t)$ and $v(x,t)$ up to a blow-up time $\Tm$.
\qed
\end{rem}

\begin{rem}
The classical solution $u=u(x,t)$ in Theorem \ref{thm:blowup}
becomes infinite  at $x = 0$ as $t\to \Tm$ and is uniformly bounded for other points in $\Omega$.
It would be interesting to know whether it is possible to extend this solution 
(in a weak sense) beyond $\Tm$.
\qed
\end{rem}

The proof of Theorem \ref{thm:blowup} is preceded by a sequence of lemmas.
We begin by  preliminary properties  of solutions on an maximal interval 
$[0,\Tm)$
of their existence. 
We skip the proof of the following lemma because such properties 
of the solutions have been already discussed   in Section \ref{sec:global}, see  inequality \eqref{v:est0}.

\begin{lemma}\label{lem0}
For all nonnegative  $u_0,v_0\in C(\bOmega)$,  problem \eqref{oeq1}-\eqref{oini} has a unique nonnegative solution on the maximal interval $[0,\Tm)$.
Moreover, 
\begin{equation}\label{v:est1} 
0\leq  v(x,t)\leq \max\left\{\|v_0\|_\infty, \frac{\kappa}{b}\right\}
\qquad \text{for all} \quad (x,t)\in \Omega \times [0,\Tm).
\end{equation}
If $\Tm<\infty$, then $\sup_{t\in [0,\Tm)}\|u(\cdot,t)\|_\infty=\infty$.
\end{lemma}

Now, we show that a constant lower bound for $v(x,t)$ leads to the blow-up of
 $u(x,t)$ in a finite time $\Tm\leq 1$.

\begin{lemma}\label{lem1}
Let $u(x,t)$ be a solution of equation \eqref{oeq1} and suppose
 that  there exists a  constant $R_0>0$ such that 
\begin{equation}\label{as1:lem1}
v(x,t)>R_0 \qquad \text{for all} \quad (x,t)\in \Omega\times [0,\Tm).
\end{equation}
If the initial condition satisfies 
\begin{equation}\label{u0:R2}
u_0(0)\geq 
\left(\frac{a}{\big(1-e^{(p-1)a}\big)F_0}\right)^{\frac{1}{p-1}},
\qquad \text{where} \quad F_0 =\inf_{v\geq R_0} f(v),
\end{equation}
then $\Tm\leq 1$.
\end{lemma}

\begin{proof}
For a fixed $v(x,t)$ with $(x,t)\in \Omega\times [0,\Tm)$, we solve equation 
\eqref{oeq1} with respect to $u(x,t)$ to obtain the following formula
for all $(x,t)\in \Omega\times [0,\Tm)$:
\begin{equation}\label{u:eq}
u(x,t)=\frac{e^{-at}}{\left(\frac{1}{u_0(x)^{p-1}}-(p-1)
\int_0^t f(v(x,s))e^{(1-p)as}\,ds\right)^{\frac{1}{p-1}}}.
\end{equation}
Thus, for $F_0=\inf_{v\geq R_0} f(v)$, equation \eqref{u:eq} 
leads to  the following lower bound 
\begin{equation}\label{u:ineq}
u(x,t)\geq \frac{e^{-at}}{\left(\frac{1}{u_0(x)^{p-1}}-
(1-e^{(1-p)at})a^{-1}F_0\right)^{\frac{1}{p-1}}}.
\end{equation}
The proof of this lemma is complete because  
 the right-hand side of  inequality \eqref{u:ineq} for $x=0$  
blows up at some  $t\leq 1$
under assumption \eqref{u0:R2}.
\end{proof}

Next, we prove that a lower bound of $v(x,t)$, required in Lemma \ref{lem1},
 is a consequence of 
a certain {\it a priori} estimate imposed on $u(x,t)$. 

\begin{lemma}\label{lem2}
Assume that $v(x,t)$ is a solution of the reaction-diffusion 
equation \eqref{oeq2} with an arbitrary function $u(x,t)$ and 
with a constant  initial condition satisfying $v_0(x)\equiv \v_0>0$.
Suppose that there are numbers $\varepsilon >0$ and 
\begin{equation}\label{alpha:ass} 
\alpha \in \left(0, \frac{2(p-1)}{p}\right)\quad  \text{if}\quad  n\geq 2
\quad \text{and}\quad 
 \alpha \in \left(0, \frac{p-1}{p}\right)\quad  \text{if}\quad  n=1
\end{equation}
such that 
\begin{equation}\label{u:as}
0<u(x,t)< \varepsilon {|x|^{-\frac{\alpha}{p-1}}} \qquad \text{for all}\quad
(x,t)\in \Omega\times [0,\Tm).
\end{equation}
Then, there is an explicit number $C_0>0$ independent of $\varepsilon$ 
{\rm (}see equation \eqref{C_0} below{\rm )} such that for all 
$\varepsilon > 0$ we have 
\begin{equation}\label{v:lower}
v(x,t)\geq  \min\left\{\v_0, \frac{\kappa}{b}\right\}-\varepsilon^p C_0
\qquad \text{for all}\quad
(x,t)\in \Omega\times [0,\Tm).
\end{equation}
\end{lemma}

\begin{proof}
We rewrite equation \eqref{oeq2} in the usual integral form ({\it cf.} \eqref{duh2})
\begin{equation}\label{v:duh}
v(t) = e^{t(\Delta -bI)t}\v_0+ \int_0^t e^{(t-s)(\Delta -bI)}\kappa\,ds
-  \int_0^t e^{(t-s)(\Delta -bI)} \big(u^pf(v)\big)(s)\,ds.
\end{equation}
Here, the function given by first two terms on the right-hand side satisfies
\begin{equation}\label{z}
z(t)\equiv  e^{t(\Delta -bI)}\v_0+ \int_0^t e^{(t-s)(\Delta -b)}\kappa\,ds
=e^{-bt}\v_0+\frac{\kappa}{b}\left(1-e^{-bt}\right)
\end{equation}
because this is an $x$-independent  solution of the problem
\begin{equation}\label{z:eq}
z_t=\Delta z -bz+\kappa, \quad z(x,0)=\v_0
\end{equation}
with the homogeneous Neumann boundary conditions. Thus
\begin{equation}\label{z:est}
z(t) \geq  \min\left\{\v_0, \frac{\kappa}{b}\right\} 
\qquad \text{for all}\quad
t\in  [0,\Tm).
\end{equation} 

Next, we recall the following well-known estimate 
\begin{equation}\label{e:est}
\left\|e^{t(\Delta-bI)}w_0\right\|_\infty\leq C_q\left(1+t^{-\frac{n}{2q}}\right)\|w_0\|_q \qquad \text{for all}\quad t>0,
\end{equation}
which is satisfied for each $w_0\in L^q(\Omega)$, 
each $q\in [1,\infty]$, and 
with a constant $C_q=C(q,n,\Omega)$ independent of $w_0$ and of $t$, see {\it e.g.}
\cite[p.~25]{R84}.

Now, we compute the $L^\infty$-norm of equation \eqref{v:duh}. Using the lower bound \eqref{z:est}, inequalities \eqref{e:est} and \eqref{v:est1}, as well as the {\it a priori} assumption on $u$ in \eqref{u:as},
 we obtain the estimate
\begin{equation}\label{v:ineq}
\begin{split}
v(x,t)&\geq   z(t)
- \int_0^t \left\|e^{(t-s)(\Delta-b)}
\big(u^pf(v)\big)(s)\right\|_\infty\,ds\\
& \geq   \min\left\{\v_0, \frac{\kappa}{b}\right\}
- \varepsilon^p C_q \left( \sup_{0\leq v\leq R_1}f(v)\right) \int_0^t
\left(1+(t-s)^{-\frac{n}{2q}}\right) \left\||x|^{-\frac{\alpha p}{p-1}}\right\|_q\,ds,
\end{split}
\end{equation} 
where the constant $R_1$ is defined in \eqref{v:est1}.
Here, we choose $q>n/2$ to have $n/(2q)<1$, which leads to the equality
$$
 \int_0^t
\left(1+(t-s)^{-\frac{n}{2q}}\right)\,ds= t+\left(1-\frac{n}{2q}\right)^{-1}
t^{1-\frac{n}{2q}}.
$$
Moreover, we assure that $q<n(p-1)/(\alpha p)$ or, equivalently, 
that $\alpha q p/(p-1)<n$ to have 
$|x|^{-\frac{\alpha p}{p-1}}\in L^q(\Omega)$.
Such a choice of $q\in [1,\infty)$ is always possible because $\max\{1, n/2\} <n(p-1)/(\alpha p)$ under our assumptions on $\alpha$ in \eqref{alpha:ass}.

Thus, for the constant
\begin{equation}\label{C_0}
C_0=C_q \left(\sup_{0\leq v\leq R_1}f(v)\right)  \left\||x|^{-\frac{\alpha p}{p-1}}\right\|_q 
\left( \Tm+\left(1-\frac{n}{2q}\right)^{-1}
\Tm^{1-\frac{n}{2q}}\right),
\end{equation}
inequality \eqref{v:ineq} implies 
the lower bound \eqref{v:lower}.
\end{proof}

Now, let us recall a classical result on the H\"older continuity of solutions 
to the inhomogeneous heat equation.

\begin{lemma}\label{lem:Holder}
Let $f\in L^\infty \big([0,T], L^q(\Omega)\big)$ with some $q>\frac{n}{2}$
and $T>0$. Denote 
$$
w(x,t) =\int_0^t e^{(t-\tau)(\Delta-bI)} f(x,\tau)\;d\tau,
$$
where $\left\{e^{t(\Delta-bI)}\right\}_{t\geq 0}$ is the semigroup of linear operators on $L^q(\Omega)$ generated by $\Delta-bI$ with
the homogeneous  Neumann boundary conditions.
There exist numbers $\beta \in (0,1)$ and $C=C>0$ 
depending on $\sup_{0\leq t\leq T}\|f(\cdot, t)\|_q$
such that 
\begin{equation}\label{Holder}
|w(x,t)-w(y,t)|\leq C|x-y|^\beta
\qquad \text{for all} \quad x,y\in \Omega
\quad\text{and}\quad t\in [0,T].
\end{equation}
\end{lemma}

\begin{proof}
Note that the function $w(x,t)$ is the solution of the
problem
$$
w_t=D\Delta w-bw +f, \qquad w(x,0)=0
$$
supplemented with  
the Neumann boundary conditions.
Hence,  estimate \eqref{Holder}
is a classical and well-known result on the H\"older continuity of solutions
to linear parabolic equations, see
{\it e.g.} \cite[Ch.~III,~\S 10]{LSU}.
\end{proof}

We apply Lemma \ref{lem:Holder} to show the H\"older continuity of $v(x,t)$.

\begin{lemma}\label{v:Holder}
Let $v(x,t)$ be a nonnegative solution of the problem
\begin{align}
&v_t  =    \Delta v-bv - u^pf(v)+\kappa&  \text{for}\quad &x\in \Omega, \quad t\in [0,\Tm)&&\label{v:eq}\\
& \frac{\partial v}{\partial n}=0
&  \text{on}\quad &\partial\Omega\times [0,\Tm), &&\label{v:N}\\
& v(x,0)  =  \v_0
&  \text{for}\quad &x\in \Omega, \quad t\in [0,\Tm),&&\label{v:0}
\end{align}
where $\v_0$ is a positive constant and $u(x,t)$ 
is a nonnegative  function. 
There exists a constant  $\alpha \in (0, 1)$ satisfying 
also \eqref{alpha:ass},  
 such that
if  the {\it a priori} estimate \eqref{u:as} for $u(x,t)$
holds true with some $\varepsilon >0$, then 
\begin{equation*}
|v(x,t)-v(y,t)|\leq \varepsilon^p C |x-y|^\alpha \qquad \text{for all}\quad
(x,t)\in \Omega\times [0,\Tm),
\end{equation*} 
where the constant $C>0$ is independent of $\varepsilon$.
\end{lemma}

\begin{proof}
As in the proof of Lemma \ref{lem2}, we use the following integral equation 
$$
v(x,t)= e^{-bt}\v_0+\frac{\kappa}{b}\left(1-e^{-bt}\right) 
-  \int_0^t e^{(\Delta -b)(t-s)} \big(u^pf(v)\big)(s)\,ds.
$$

Suppose that $u(x,t)$ satisfies the {\it  a priori}  estimate  \eqref{u:as}
with a certain number $\alpha\in (0,1)$ satisfying relations \eqref{alpha:ass}.
Since $f(v)\in L^\infty \big(\Omega\times [0,\Tm)\big)$ by \eqref{v:est1}
and since $|u^p(x,t)|\leq \varepsilon ^p |x|^{-\alpha p/(p-1)}$ by assumption \eqref{u:as},
we obtain 
$$
u^pf(v)\in   L^\infty \big([0,\Tm), L^q(\Omega)\big)\qquad 
\text{for some} \quad q>n/2,
$$
see the proof of Lemma \ref{lem2}.
Thus, by  Lemma~\ref{lem:Holder}, there exist constants $C>0$ and 
$\beta\in (0,1)$, independent of $\varepsilon $  such that
 $|v(x,t)-v(y,t)|\leq \varepsilon^p C |x-y|^\beta$
for all $x,y\in\Omega$ and $t\in [0,\Tm)$. Without loss of generality, we can 
 assume that $\beta$ satisfies the conditions in~\eqref{alpha:ass} (we can always take it smaller).

The proof is completed, if $\beta\geq \alpha$.
On the other hand, if  $\beta<\alpha$,
 we suppose  the {\it a priori} estimate 
$0\leq u(x,t) <\varepsilon |x|^{-\beta/(p-1)}$ for all $x\in\Omega$ and $t\in [0,\Tm)$. Thus, 
 there exists a constant 
$C=C(\alpha, \beta, p, \Omega)>0$ such that 
$$
0\leq u(x,t) <\varepsilon |x|^{-\frac{\beta}{p-1}}
=\varepsilon |x|^{-\frac{\alpha}{p-1}} |x|^{\frac{\alpha-\beta}{p-1}}
\leq C\varepsilon |x|^{-\frac{\alpha}{p-1}}.
$$
Hence, repeating  the  reasoning in the preceding 
paragraph of  this proof, we obtain again the estimate 
$|v(x,t)-v(y,t)|\leq \varepsilon^p C |x-y|^\beta$
for all $x,y\in\Omega$ and $t\in [0,\Tm)$ with a modified constant $C>0$, but still independent of $\varepsilon>0$.
\end{proof}

\begin{proof}[Proof of Theorem \ref{thm:blowup}]
By Lemmas \ref{lem1} and \ref{lem2}, it suffices to show the {\it a priori} estimate
\begin{equation}\label{u:as2}
0<u(x,t)< \varepsilon {|x|^{-\frac{\alpha}{p-1}}} \qquad \text{for all}\quad
(x,t)\in \Omega\times [0,\Tm)
\end{equation}
with  $\Tm\leq 1$, under the assumption that $\varepsilon>0$ is sufficiently small. 

By assumption \eqref{as1:u0}
(see Remark \ref{rem:u0}), we have  $0<u_0(x)< \varepsilon{|x|^{-\frac{\alpha}{p-1}}}$
for all $x\in \Omega$, hence, 
by  a continuity argument, inequality \eqref{u:as2} is satisfied on a certain initial time interval.
Suppose that there exists $T_1\in (0,1)$ such that the solution of problem \eqref{oeq1}-\eqref{oini} exists on the interval $[0,T_1]$  and satisfies
\begin{align}
&\sup_{x\in\Omega} |x|^{\frac{\alpha}{p-1}}u(x,t)<\varepsilon\qquad\text{for all}\quad  t<T_1,\label{u<1}\\
&\sup_{x\in\Omega} |x|^{\frac{\alpha}{p-1}}u(x,t)=\varepsilon \qquad\text{for}\quad  t=T_1.
\label{u=1}
\end{align}
From now on, we are going to use the explicit formula for $u(x,t)$ in \eqref{u:eq} and the H\"older regularity of $v(x,t)$ from Lemma \ref{lem:Holder}. 
First, we estimate the denominator of the fraction  in \eqref{u:eq} using
 assumption  \eqref{as1:u0} as follows
\begin{equation}\label{denom}
\begin{split}
\frac{1}{u_0(x)^{p-1}}-&(p-1) \int_0^t f(v(x,s))e^{(1-p)as}\,ds \\
&\geq 
2\varepsilon^{p-1}|x|^\alpha + \frac{1}{u_0(0)^{p-1}}-(p-1) \int_0^t f(v(0,s))e^{(1-p)as}\,ds\\
&\quad +(p-1) \int_0^t \big(f(v(0,s))-f(v(x,s))\big)e^{(1-p)as}\,ds.
\end{split}
\end{equation}
By the definition of $\Tm$ and formula  \eqref{u:eq}, we immediately 
obtain 
\begin{equation}\label{second}
\frac{1}{u_0(0)^{p-1}}-(p-1) \int_0^t f(v(0,s))e^{(1-p)as}\,ds>0 \qquad 
\text{for all} \quad t\in [0,\Tm).
\end{equation}
Next, we use our hypothesis \eqref{u<1} and \eqref{u=1} together with 
 the H\"older continuity of $v(x,t)$ established in Lemma \ref{v:Holder} to find  constants $C>0$ and $\alpha \in (0,1)$ (satisfying also \eqref{alpha:ass}),
the both 
independent of $\varepsilon \geq 0$, such that
$$
\big|f(v(0,s))-f(v(x,s))\big|\leq C\varepsilon^p |x|^\alpha\qquad \text{for all} \quad t\in [0,T_1].
$$
Hence, since $T_1\leq \Tm\leq 1$, we obtain the following  bound for 
the last term on the right-hand side of \eqref{denom}:
\begin{equation}\label{third}
(p-1) \int_0^t \big|f(v(0,s))-f(v(x,s))\big|e^{(1-p)as}\,ds
\leq \varepsilon^p  Ca^{-1} |x|^{\alpha}
\end{equation}
for all $(x,t)\in \Omega\times [0,T_1]$.
Consequently,  applying inequalities \eqref{second}
and \eqref{third} in \eqref{denom}
we obtain the lower bound for the denominator in \eqref{u:eq}
\begin{equation}\label{denom1}
\frac{1}{u_0(x)^{p-1}}-(p-1) \int_0^t f(v(x,s))e^{(1-p)as}\,ds \geq 
\big(2\varepsilon^{-(p-1)}-\varepsilon^p Ca^{-1}\big)|x|^\alpha
\end{equation}
for all $(x,t)\in \Omega\times [0,T_1]$.
Finally, we choose $\varepsilon>0$ so small that 
$2\varepsilon^{-(p-1)}-\varepsilon^p Ca^{-1}>\varepsilon^{-(p-1)}$
and we substitute estimate \eqref{denom1} in equation \eqref{u:eq} to obtain
$$
0<u(x,t)\leq  
\frac{e^{-at}}{\Big(\big(2\varepsilon^{-(p-1)}-\varepsilon^p Ca^{-1}\big)|x|^\alpha\Big)^{\frac{1}{p-1}}} 
< \frac{\varepsilon}{|x|^{\frac{\alpha}{p-1}}} \qquad \text{for all}\quad
(x,t)\in \Omega\times [0,T_1].
$$
This inequality for $t=T_1$ contradicts our hypothesis \eqref{u=1}.

Thus, estimate \eqref{u:as2}  holds true on the whole interval 
$[0,\Tm)$. Then, by Lemma \ref{lem2}, the function $v(x,t)$ is bounded from below by a  constant
$R_0= \min\left\{\v_0, \frac{\kappa}{b}\right\}-\varepsilon^p C_0$ which is positive provided $\varepsilon>0$ is sufficiently small.
Finally, Lemma \ref{lem1} implies that $u(x,t)$ blows up at $x=0$ and 
at certain $\Tm\leq 1$, if $u_0(0)$ satisfies inequality \eqref{as3:u0}.
\end{proof}

\section*{Acknowledgments}
A.~Marciniak-Czochra was supported by European Research Council Starting Grant No 210680 ``Multiscale mathematical modelling of dynamics of structure formation in cell systems'' and Emmy Noether Programme of German Research Council (DFG). %
G.~Karch was  supported
by the NCN grant No.~2013/09/B/ST1/04412. 
K.~Suzuki  acknowledges  JSPS the Grant-in-Aid for Young Scientists (B) 23740118.


%

\def\cprime{$'$}

\end{document}